\documentclass[11pt, reqno]{amsart}       
\usepackage{graphicx}

\usepackage{amsmath}

\usepackage{amssymb}
\usepackage{pifont}
\newtheorem{thm}{Theorem}
\newtheorem*{thm1}{Theorem}
\newtheorem{lem}{Lemma}

\newtheorem{cor}{Corollary}
\theoremstyle{definition}

\newtheorem*{rem}{Remarks}

\newtheorem*{rem2}{Remark}
\newtheorem*{example}{Example}

\newcommand{\vertiii}[1]{{\left\vert\kern-0.25ex\left\vert\kern-0.25ex\left\vert
		#1 \right\vert\kern-0.25ex\right\vert\kern-0.25ex\right\vert}}

\def \lim   {\text {\rm lim}}

\begin{document}
	\title{Uncertainty principles on $C^*$-algebras}

	\author{Saptak Bhattacharya}
	
	\address{Indian Statistical Institute\\
		New Delhi 110016\\
		India}
	\email{saptak21r@isid.ac.in}
	
	
	
	
	\begin{abstract}In this paper we prove some uncertainty bounds for commutators and anti-commutators of observables in a $C^*$-algebra. We give a short, elementary proof of Robertson's Standard Uncertaity Principle in this setting. We also prove some other uncertainty relations for which the lower bound doesn't vanish for any number of observables.\end{abstract}
	
	
	\date{}
	\maketitle
	\vspace{2mm}
	\section{Introduction}
	Let $\mathcal{A}$ be a unital $C^*$-algebra and let $\phi:\mathcal{A}\to\mathbb{C}$ be a state. Let $\{x_j\}_{j=1}^n$ be self-adjoint elements of $\mathcal{A}$. We normalize them so that $\phi(x_j)=0$ for all $1\leq j\leq n$. The {\it covariance matrix} of the $n$-tuple $\tilde{x}=(x_j)_{j=1}^n$ is the $n\times n$ positive matrix given by \[\mathrm{Cov}_{\phi}(\tilde{x})=\big(\mathrm{Re}\hspace{1mm}\phi(x_i x_j)\big).\] This can also be written as \[\mathrm{Cov}_{\phi}(\tilde{x})={M+M^T}\] where \[M=\frac{1}{2}\big(\phi(x_ix_j)\big)\label{e1}\tag{1}\] is a Gram matrix. Note that \[M-M^T=\frac{1}{2}\big(\phi[x_i, x_j]\big)\] and \[M+M^T=\mathrm{Cov}_{\phi}(\tilde{x})=\frac{1}{2}\big(\phi\{x_i,x_j\}\big)\] where $[x_i, x_j]=x_i x_j-x_j x_i$ is the commutator and $\{x_i, x_j\}=x_ix_j+x_jx_i$ is the anti-commutator of $x_i$ and $x_j$. In this form, Robertson's {\it Standard Uncertainty Principle}  (see \cite{rob}) can be stated as : \[\mathrm{det}|M-M^T|\leq\mathrm{det}\hspace{1mm}\mathrm{Cov}_{\phi}(\tilde{x})\label{e2}\tag{2}\] Note that for two self-adjoint elements $x$ and $y$, this reduces to the well-known Schr\"{o}dinger-Heisenberg uncertainty principle (\cite{sh}) :
	\[\phi(x^2)\phi(y^2)-|\mathrm{Re}\hspace{1mm}\phi(yx)|^2\geq\frac{1}{4}|\phi[x, y]|^2.\]

	Though Robertson's SUP generalizes Heisenberg's uncertainty principle for $n$ observables, it has a downside. Note that $i(M-M^T)$ is a real skew-symmetric matrix, and therefore, if $n$ is odd, its determinant vanishes. Thus, for an odd number of observables, inequality \eqref{e2} fails to capture any quantum phenomenon arising from non-commutativity. This prompted research on uncertainty relations which do not have this drawback. Andai in \cite{an} and Gibilisco, Imparato and Isola in \cite{gi} proved the following :
	\vspace{2mm}
	
	\begin{thm1}Let $A$ be a positive definite density matrix. Let $\{H_j\}_{j=1}^k$ be Hermitian and let $f:(0,\infty)\to\mathbb{R}$ be a positive operator monotone function satisfying $f(1)=1$ and $xf(\frac{1}{x})=f(x)$ for all $x\in (0,\infty)$. Then \[\mathrm{det}\hspace{1mm}\big(\frac{f(0)}{2}\gamma_f([A, H_i],[A, H_j])\big)\leq\mathrm{det}\hspace{1mm}\big(\mathrm{Re}\hspace{1mm}\mathrm{tr}(AH_i H_j)\big).\label{e3}\tag{3}\]\end{thm1} Here $\gamma_f$ denotes the quantum Fisher information metric (see \cite{pet, pet1}) associated to $f$. Inequality \ref{e3} is also known as the {\it Dynamical Uncertainty Principle}. This was later generalized to arbitrary covariances by Gibilisco, Hiai and Petz in \cite{ghp}. The question arises, can we have variants of the standard uncertainty principle involving the commutators $[x_i, x_j]$ with non-vanishing lower bound regardless of the number of observables? This paper is an attempt in that direction. The sum and difference of positive matrices play a key role here, utilizing which we shall prove some uncertainty bounds. Lastly, for $k$ Hermitian matrices $\{H_j\}_{j=1}^k$, we discuss the problem of comparing the block matrices $\big([H_i, H_j]\big)$ and $\big(\{H_i, H_j\}\big)$ without composing a state entrywise, and prove an interesting trace inequality. It is to be noted that in this case, we cannot write both sides as the sum and difference of two positive matrices, prompting the use of other techniques. 
	\section{Main results}
	
	We first give a short proof of  Robertson's SUP. For that we need a lemma.
	\begin{lem}\label{l1} Let $A, B$ be two positive matrices. Then \[ \mathrm{det}(A-B)^2\leq  \mathrm{det}(A+B)^2.\]
	\end{lem}
	\begin{proof}Suffices to prove for positive definite $A$ and $B$. If $A$ and $B$ commute, the inequality is obvious. Otherwise, we note that \[\begin{aligned}&\textrm{det}(A-B)^2\\&=\textrm{det}\big(A^{1/2}(I-A^{-1/2}BA^{-1/2})A^{1/2}\big)^{2}\\&=\textrm{det}A^2\hspace{1mm}\textrm{det}(I-A^{-1/2}BA^{-1/2})^{2} \\&\leq \textrm{det} A^2\hspace{1mm}\textrm{det}(I+A^{-1/2}BA^{-1/2})^{2} \\&=\textrm{det}(A+B)^2.\end{aligned}\]\end{proof}
	\begin{rem2}Note that the operator inequality $(A-B)^2\leq(A+B)^2$ need not be true for positive $A$ and $B$ unless $A$ and $B$ commute.\end{rem2}
	
	As an immediate consequence we have Robertson's SUP.
	\begin{thm}\label{t1} Let $\{x_i\}_{i=1}^n$ be self-adjoint elements in a $C^*$-algebra $\mathcal{A}$. Let $\phi:\mathcal{A}\to\mathbb{C}$ be a state. Then, \[\mathrm{det}\big(\frac{1}{2}\phi[x_i, x_j])^2\leq\mathrm{det}\hspace{1mm}\mathrm{Cov}_{\phi}(\tilde{x})^2.\]\end{thm}
	\begin{proof} Consider the positive matrix $M$ as given in \eqref{e1} and apply lemma \ref{l1}. 
		
	\end{proof}
	\vspace{3mm}
	
	Before we state the next lemma, let us have a quick review of the matrix geometric mean. Given two positive matrices $A$ and $B$, their geometric mean is given by \[A\# B = A^{1/2}(A^{-1/2}BA^{-1/2})^{1/2}A^{1/2}.\] Let $\mathbb{P}(n)$ be the set of all $n\times n$ positive definite matrices equipped with the Riemannian metric \[\langle H, K \rangle_A = \mathrm{tr\hspace{1mm}}(A^{-1}HA^{-1}K)\label{e4}\tag{4}\] for all $A\in\mathbb{P}(n)$, and Hermitian $H$ and $K$. It can be shown that for $A, B\in\mathbb{P}(n)$, $A\# B$ is precisely the Riemannian mean of $A$ and $B$ with respect to the metric \eqref{e4}. Infact, the geodesic joining $A$ and $B$ is given explicitly by $$t\to A^{1/2}(A^{-1/2}BA^{-1/2})^tA^{1/2}$$ where $t\in[0,1]$.
	An important variational characterization of $A\# B$ is given by \[A\#B = \mathrm{max}\{X:X^*=X, \begin{pmatrix}A & X\\X & B\end{pmatrix}\geq O\}\label{e5}\tag{5}.\] We shall be using this in the proof of our next result.
	\\\\
	The geometric mean is extremely important and finds wide applications in matrix inequalities, quantum information, and even in medical imaging. See \cite{rb, rb1, ll, hia, jen, bh2} for more details.
	
	\begin{lem}\label{l2} Let $A, B$ be positive matrices. Then there exists a unitary $U$ such that \[|A-B|\leq (A+B)\# U(A+B)U^*.\] \end{lem}
	\begin{proof}Note that \[\begin{pmatrix}A+B & A-B\\A-B & A+B\end{pmatrix}=\begin{pmatrix}1 & 1\\1 & 1\end{pmatrix}\otimes A +\begin{pmatrix}1 & -1\\-1 & 1\end{pmatrix}\otimes B\] which is positive. Let \[A-B=U|A-B|\] be the polar decomposition of $A-B$. Note that $U$ commutes with $A-B$ and hence, $|A-B|$ since $A-B$ is Hermitian. Observe that
		\vspace{2mm}
		
		\[\begin{pmatrix}A+B & |A-B|\\|A-B| & U(A+B)U^*\end{pmatrix}=\begin{pmatrix} I & O\\ O & U\end{pmatrix}\begin{pmatrix}A+B & A-B\\A-B & A+B\end{pmatrix}\begin{pmatrix} I & O\\ O & U^*\end{pmatrix}.\]
		\vspace{2mm}
		
		Hence, the matrix \[\begin{pmatrix}A+B & |A-B|\\|A-B| & U(A+B)U^*\end{pmatrix}\] is positive and therefore, \[|A-B|\leq (A+B)\# U(A+B)U^*\] by \eqref{e5}.
	\end{proof}

	As an corollary, we have :
	
	\begin{cor}\label{c1} Let $\{x_j\}_{j=1}^n$ be self-adjoint elements in a $C^*$-algebra $\mathcal{A}$ and let $\phi:\mathcal{A}\to\mathbb{C}$ be a state. Then for any unitarily invariant norm $|||.|||$ on $M_n(\mathbb{C})$, \[|||\big(\frac{1}{2}\phi[x_i, x_j]\big)|||\leq |||\mathrm{Cov}_{\phi}(\tilde{x})|||.\]\end{cor}
	\begin{proof} Recall that $$M+M^T=\mathrm{Cov}_{\phi}(\tilde{x})$$ and $$M-M^{t}=\big(\frac{1}{2}\phi[x_i, x_j]\big).$$ By lemma \ref{l2} there exists a unitary $U$ such that \[\begin{aligned}&|M-M^T|\\&\leq (M+M^T)\# U(M+M^T)U^*\\&\leq\frac{(M+M^T+U(M+M^T)U^*)}{2}.\end{aligned}\] Thus, \[\begin{aligned}&|||M-M^T|||\\&=|||\hspace{1mm}|M-M^T|\hspace{1mm}|||\\&\leq|||M+M^T|||\\&=|||\mathrm{Cov}_{\phi}(\tilde{x})|||.\end{aligned}\]
	\end{proof}
	
	\begin{rem2} Taking the Frobenius norm $||.||_2$ in corollary \ref{c1}, we have \[\frac{1}{2}\sum_{i\textless j}|\phi[x_i, x_j]|^2\leq ||\mathrm{Cov}_{\phi}(\tilde{x})||_2^2\leq\big(\mathrm{tr}\hspace{1mm}\mathrm{Cov}_{\phi}(\tilde{x})\big)^2\leq\big[\sum_j\mathrm{Var}(x_j)\big]^2\label{e6}\tag{6}\]\end{rem2} This gives a non-vanishing lower bound of the sum of variances in terms of the Lie brackets of the observables, leading to an uncertainty relation.
	
	Given a positive matrix $A=(a_{ij})$ it is well known that \[\mathrm{det}\hspace{1mm}A\leq\prod_j a_{jj}.\] Using this, we see that \[\mathrm{det}|M-M^T|\leq\prod_j\langle |M-M^T|e_j, e_j\rangle.\] The right hand side is non-vanishing and we would like to have an upper bound for it in terms of $\textrm{Cov}_{\phi}(\tilde{x})$. We need the following inequality : 
	\begin{thm}\label{t2} Let $A, B\in M_n(\mathbb{C})$ be positive matrices. Then \[\big[\prod_j\langle |A-B|e_j, e_j\rangle\big]^{2/n}\leq \frac{\mathrm{tr}\hspace{1mm}A+\mathrm{tr}\hspace{1mm} B}{n}\big[\prod_j\langle(A+B)e_j, e_j\rangle\big]^{1/n}.\]\end{thm}
	\begin{proof}By lemma \ref{l2} there exists a unitary $U$ such that \[|A-B|\leq(A+B)\# U(A+B)U^*.\] Taking tensor powers, \[\otimes^n |A-B|\leq[\otimes^n(A+B)]\#[\otimes^n U(A+B)U^*].\label{e7}\tag{7}\] Let $\{e_j\}_{j=1}^n$ be the standard basis vectors of $\mathbb{C}^n$. Consider the state $\xi: \otimes^n M_n(\mathbb{C})\to\mathbb{C}$ given by \[\xi(X)=\langle X(\otimes_{j=1}^n e_j), \otimes_{j=1}^ne_j\rangle.\] Now, \[\begin{aligned}&\Pi_j\langle|A-B|e_j, e_j\rangle^2\\&=\xi\big(\otimes^n|A-B|\big)^2\\&\leq \xi\big(\otimes^n (A+B)\big)\hspace{1mm}\xi\big(\otimes^n U(A+B)U^*\big)\\&\leq\prod_j\langle(A+B)U^*e_j, U^*e_j\rangle\hspace{1mm}\prod_j\langle(A+B)e_j, e_j\rangle\\&\leq\big[\frac{\mathrm{tr}\hspace{1mm}(A+B)}{n}\big]^n\hspace{1mm} \prod_j\langle(A+B)e_j, e_j\rangle .\end{aligned}\] Taking powers of $1/n$, we are done.
	\end{proof}
	Applying Theorem \ref{t2} to the positive matrices $M$ and $M^T$, we have the following :
	\begin{cor}\label{c2}Let $\{x_j\}_{j=1}^n$ be self-adjoint elements in a $C^*$-algebra $\mathcal{A}$ and let $\phi:\mathcal{A}\to\mathbb{C}$ be a state. Let $\sigma(x_j)=\sqrt{\phi(x_j^2)}$ be the standard deviation of $x_j$. Then \[\frac{1}{4}\prod_{k=1}^n\langle |(\phi[x_i, x_j])|e_k,e_k\rangle^{2/n}\leq \big[\frac{\sum\sigma^2(x_k)}{n}\big] \big[\prod_k\sigma^2(x_k)\big]^{1/n}.\label{e8}\tag{8} \]\end{cor}
	From inequality \ref{e8} we also get \[\frac{1}{2}\prod\langle |(\phi[x_i, x_j])|e_k,e_k\rangle^{1/n}\leq \frac{\sum\sigma^2(x_k)}{n}.\label{e9}\tag{9}\] This gives a non-vanishing lower bound for the arithmetic mean of the variances and hence, an uncertainty relation. Note that for two self-adjoint elements $x_1$ and $x_2$, the matrix $$\big|\big(\phi[x_i, x_j]\big)\big|$$ is scalar and hence, the left hand side becomes the same as $$\mathrm{det}\hspace{1mm}\big|\big(\phi[x_i, x_j]\big)\big|^{1/n}.$$ This is not the case for several self-adjoint elements, which makes it complicated. 
	\vspace{2mm}
	
	The next inequality gives another non-zero lower bound in terms of the entries $\phi[x_i, x_j]$.
	
	\begin{thm}\label{t3} Let $\mathcal{A}$ be a $C^{*}$-algebra and let $\{x_j\}_{j=1}^n$ be self-adjoint elements of $\mathcal{A}$. Let $\phi:\mathcal{A}\to\mathbb{C}$ be a state. Then \[\frac{1}{4}\prod_k\big[\sum_{i\neq k}|\phi[x_i, x_k]|^2\big]^{1/n}\leq||\mathrm{Cov}_{\phi}(\tilde{x})||\hspace{1mm}\big[\frac{\sum \sigma^2(x_k)}{n}\big]^{1/2}\big[\prod_k\sigma(x_k)\big]^{1/n}.\]\end{thm}
	\begin{proof} For each $k$, \[\begin{aligned}& \frac{1}{4}\sum_{i\neq k}|\phi[x_i, x_k]|^2\\&=\langle (M-M^T)^2e_k e_k\rangle\\&\leq ||M-M^T||\hspace{1mm}\langle |M-M^T|e_k, e_k\rangle\\&\leq||\mathrm{Cov}_{\phi}(\tilde{x})||\hspace{1mm}\langle |M-M^T|e_k, e_k\rangle\end{aligned}\] where the last inequality follows from corollary \ref{c1}. Taking geometric mean on both sides and using corollary \ref{c2}, we are done.\end{proof}
	
	\vspace{3mm}
	
	Given Hermitian matrices $\{H_j\}_{j=1}^n$, we now try to compare the block matrices $\big([H_i,H_j]\big)$ and $\big(\{H_i,H_j\}\big)$ instead of composing a state entrywise. Note that the block matrix $\big(H_j H_i\big)$ is not necessarily positive even though $\big(H_i H_j\big)$ always is. This is demonstrated by the following example :
	
	\begin{example}Let $\theta\in\mathbb{R}$ and let \[H_1=\begin{pmatrix}\frac{\sin\hspace{1mm} \theta}{2} & \cos\hspace{1mm} \theta\\ \cos\hspace{1mm} \theta & \sin\hspace{1mm} \theta\end{pmatrix}\] and \[H_2=\begin{pmatrix}\cos\hspace{1mm} \theta & \sin\hspace{1mm} \theta \\ \sin\hspace{1mm} \theta & \frac{\cos\hspace{1mm} \theta}{2}\end{pmatrix}.\] Let $e_1, e_2$ be the standard basis vectors of $\mathbb{C}^2$. Then the matrix \[\big(\langle H_i e_j, H_j e_i\rangle\big)\] has negative determinant and hence, is not positive.\end{example}
	\vspace{3mm}
	
	Hence, unlike the previous results, we cannot use the idea of writing both sides as the sum and difference of two positive matrices. However, we can still obtain an inequality with the trace norm such that the left hand side vanishes if and only if all the observables commute. For two Hermitian matrices $H_1$ and $H_2$, this takes the form \[\big|\big|\begin{pmatrix}O & [H_1, H_2] \\ [H_2, H_1] & O\end{pmatrix}\big|\big|_1\leq2(\mathrm{tr}\hspace{1mm}H_1^2+\mathrm{tr}\hspace{1mm}H_2^2).\] This is easy to prove, since, \[\begin{aligned}&\big|\big|\begin{pmatrix}O & [H_1, H_2] \\ [H_2, H_1] & O\end{pmatrix}\big|\big|_1\\&=2\hspace{1mm}\mathrm{tr}\hspace{1mm}\big|[H_1,H_2]\big|\\&\leq 2\big(||H_1 H_2||_1+||H_2 H_1||_1\big)\\&\leq 4\sqrt{\mathrm{tr}\hspace{1mm}H_1^2\hspace{1mm} \mathrm{tr}\hspace{1mm}H_2^2}\\&\leq2(\mathrm{tr}\hspace{1mm}H_1^2+\mathrm{tr}\hspace{1mm}H_2^2).\end{aligned}\]
	
	The next theorem generalizes this to $k$ Hermitian matrices $\{H_j\}_{j=1}^k$.
	
	\begin{thm}\label{t4} Let $\{H_j\}_{j=1}^k$ be $n\times n$ Hermitian matrices. Then \[\big|\big|\big([H_i, H_j]\big)\big|\big|_1\leq (k-1)\mathrm{tr}\hspace{1mm}\big(\{H_i, H_j\}\big).\]\end{thm}
	\begin{proof}Let $$A=\big([H_i, H_j]\big).$$ Note that $A$ is Hermitian and \[||A||_1=\mathrm{tr}\hspace{1mm}(A^2)^{1/2}.\] The matrix obtained by pinching $A^2$ along the main diagonal is given by \[D=\begin{pmatrix}-\sum_{j\neq 1}[H_1, H_j]^2 & &\\ & \ddots\\& & -\sum_{j\neq n}[H_n, H_j]^2\end{pmatrix}.\] Note that \[D=\frac{1}{2\pi}\int_{0}^{2\pi}U(\theta)A^2U(\theta)^* d\theta\] where \[U(\theta)=\begin{pmatrix}I & & &\\ & e^{i\theta}I & &\\& & \ddots &\\ & & & e^{i(k-1)\theta}I.\end{pmatrix}\] By operator concavity of the square root, \[D^{1/2}\geq \frac{1}{2\pi}\int_{0}^{2\pi}U(\theta)|A|U(\theta)^* d\theta.\] Taking traces, we get the inequality \[||A||_1\leq\sum_i\mathrm{tr}\hspace{1mm}\big(\sum_{j} \big|[H_i, H_j]\big|^2)^{1/2}\big).\label{e10}\tag{10}\] Now, for each $i$, consider the block matrix \[\tilde{H}_i=\begin{pmatrix}{}[H_i, H_1] &O  &\cdots &O\\ \vdots & & &\\{}[H_i, H_k] & O &\cdots  &O \\\end{pmatrix}.\] Note that inequality $\eqref{e10}$ now becomes \[||A||_1\leq \sum_i ||\tilde{H}_i||_1.\] Let \[Y_i = \begin{pmatrix}H_i H_1 & O &\cdots &\cdots & O\\\vdots & & & & &\\ H_i H_{i-1 } & & & &\\\\ O & & & &\\\\ H_i H_{i+1} & & & &\\\vdots & & & &\\ H_i H_k & O &\cdots &\cdots & O\\\end{pmatrix}\] and \[Z_i = \begin{pmatrix}H_1 H_i & O &\cdots &\cdots & O\\\vdots & & & & &\\ H_{i-1} H_{i} & & & &\\\\ O & & & &\\\\ H_{i+1} H_{i} & & & &\\\vdots & & & &\\ H_k H_i & O &\cdots &\cdots & O\\\end{pmatrix}\] where the $O$ in the first colum is at the $i^{\mathrm{th}}$ position. Note that $\tilde{H}_i = Y_i - Z_i$. Inequality $\eqref{e10}$ now gives \[\begin{aligned}&||A||_1\\&\leq\sum_i ||Y_i-Z_i||_1\\&\leq \sum_i ||Y_i||_1+||Z_i||_1.\end{aligned}\label{e11}\tag{11}\] Consider the $k\times k$ block matrices \[\tilde{H}=\begin{pmatrix}H_1 & O &\cdots &\cdots  & O &\\\vdots & & & & \\ H_{i-1}  & & & &\\ \\ O & & & &\\ \\ H_{i+1} & & & &\\\vdots & & & &\\ H_k & O &\cdots &\cdots & O &\\\end{pmatrix}, \hspace{2mm}\tilde{Y}_i=\begin{pmatrix}H_i & & & & & &\\ &\ddots & & & & &\\ & & H_i &  &  & &\\ & & & O & & &\\ & & & & H_i & &\\ & & & & &\ddots &\\ & & & & & & H_i\end{pmatrix}\] and \[\tilde{K}_i=\begin{pmatrix}H_i & O &\cdots &\cdots  & O &\\\vdots & & & & \\ H_{i}  & & & &\\ \\ O & & & &\\ \\ H_{i} & & & &\\\vdots & & & &\\ H_i & O &\cdots &\cdots & O &\\\end{pmatrix}, \hspace{2mm}\tilde{Z}=\begin{pmatrix}H_1 & & & & & &\\ &\ddots & & & & &\\ & & H_{i-1} &  &  & &\\ & & & O & & &\\ & & & & H_{i+1} & &\\ & & & & &\ddots &\\ & & & & & & H_k\end{pmatrix}.\] Here, $O$ is the $i^{\mathrm{th}}$ entry in the first column of $\tilde{H}$ and $\tilde{K}_i$, and the $i^{\mathrm{th}}$ diagonal entry of $\tilde{Y}_i$ and $\tilde{Z}$. Note that $Y_i = \tilde{H}\tilde{Y}_i$ and $Z_i = \tilde{K}_i\tilde{Z}$. By Cauchy-Schwarz and A.M-G.M, \[\begin{aligned}&||Y_i||_1\\&\leq||\tilde{H}||_2 ||\tilde{Y}_i||_2\\&\leq \frac{||\tilde{H}||_2^{2}+||\tilde{Y}_i||_2^{2}}{2}\\&\leq\frac{(k-1)\hspace{1mm}\mathrm{tr}\hspace{1mm}H_i^2+\sum_{j\neq i}\mathrm{tr}\hspace{1mm}H_j^2}{2}.\end{aligned}\label{e12}\tag{12}\] Similarly, \[\begin{aligned}&||Z_i||_1\\&\leq\frac{(k-1)\hspace{1mm}\mathrm{tr}\hspace{1mm}H_i^2+\sum_{j\neq i}\mathrm{tr}\hspace{1mm}H_j^2}{2}.\end{aligned}\label{e13}\tag{13}\] By $\eqref{e11}$, $\eqref{e12}$ and $\eqref{e13}$, \[\begin{aligned}&||A||_1\\&\leq(k-1)\sum_i\mathrm{tr}\hspace{1mm}H_i^2+ \sum_i \big[\sum_{j}\mathrm{tr}\hspace{1mm}H_j^2 -\mathrm{tr}\hspace{1mm}H_i^2\big]\\&=2(k-1)\sum_i \mathrm{tr}\hspace{1mm}H_i^2\\&=(k-1)\hspace{1mm}\mathrm{tr}\hspace{1mm}\big(\{H_i, H_j\}\big).\end{aligned}\]
		
	\end{proof}
	
	\begin{rem}
	\end{rem}
	
	\begin{enumerate}
		\item Note that Theorem \ref{t4} gives a refinement of corollary \ref{c1} when $\phi$ is the normalized trace on $M_n(\mathbb{C})$ and $|||.|||$ is the trace norm. In this case, the left hand side of the inequality in corollary \ref{c1} vanishes, but Theorem \ref{t4} gives a non-vanishing lower bound.
		\vspace{3mm}
		
		\item Let $\mathcal{M}$ be a von-Neumann algebra with a faithful tracial state $\tau$. Given $x\in\mathcal{M}$, its trace norm is given by \[||x||_1=\tau(|x|)\] for all $x\in\mathcal{M}$. We can lift $\tau$ to $\tau_k : M_k(\mathcal{M})\to\mathbb{C}$ via \[\tau_k (x_{ij})=\sum_{i}\tau(x_{ii}).\] Note that $\tau_k$ is positive, faithful and tracial. Let $\{x_j\}_{j=1}^k$ be self-adjoint elements of $\mathcal{M}$. In this context, Theorem \ref{t4} reads \[\big|\big|\big([x_i, x_j]\big)\big|\big|_1\leq (k-1)\tau \big({x_i, x_j}\big).\] Indeed, the same proof holds, since the only properties of the trace being used there are invariance under unitary conjugation, triangle inequality and Cauchy-Schwarz inequality, all of which hold in the von-Neumann algebra setting (see \cite{sil, nel}).\end{enumerate}
	
	\bibliographystyle{amsplain}
		
\end{document}